\documentclass{amsart}
\usepackage{amsfonts}

\newtheorem{theorem}{Theorem}
\theoremstyle{plain}

\newtheorem{definition}{Definition}

\numberwithin{equation}{section}
\input{tcilatex}

\begin{document}
\title[Tribonacci Quaternions]{On some Properties of Tribonacci Quaternions}
\author{Ilker Akkus}
\address[I. Akkus and G. K\i z\i laslan]{K\i r\i kkale University,
Department of Mathematics, Faculty of Science and Arts, 71450 K\i r\i kkale
Turkey}
\email[I. Akkus]{iakkus.tr@gmail.com}
\author{Gonca K\i z\i laslan}
\email[G. K\i z\i laslan]{goncakizilaslan@gmail.com}
\thanks{This paper is in final form and no version of it will be submitted
for publication elsewhere.}

\begin{abstract}
In this paper, we give some properties of the Tribonacci and
Tribonacci-Lucas quaternions and obtain some identities for them.
\end{abstract}

\maketitle

\section{Introduction}

Quaternions are fundamental objects of various parts of mathematics. They
have applications in both theoretical and applied mathematics such as group
theory, computer science and even also physics, see \cite{girard,kou,startek}%
. Let $\mathcal{H}$ be the real division quaternion algebra. A natural basis
of this algebra is formed by the elements $\mathbf{1},$ $\mathbf{i},$ $%
\mathbf{j},$ $\mathbf{k}$ where $\mathbf{i}^{2}=\mathbf{j}^{2}=\mathbf{k}%
^{2}=\mathbf{ijk}=-1$. So all quaternions are of the form%
\begin{equation*}
q=a_{0}+\mathbf{i}a_{1}+\mathbf{j}a_{2}+\mathbf{k}a_{3},
\end{equation*}%
where the coefficients $a_{n}$, $0\leq n\leq 3$ are all real. The
multiplication table for the basis of $\mathcal{H}$ is%
\begin{equation*}
\begin{tabular}{|r||r|r|r|r|}
\hline
$\cdot $ & $\mathbf{1}$ & $\mathbf{i}$ & $\mathbf{j}$ & $\mathbf{k}$ \\ 
\hline\hline
$\mathbf{1}$ & $\mathbf{1}$ & $\mathbf{i}$ & $\mathbf{j}$ & $\mathbf{k}$ \\ 
\hline
$\mathbf{i}$ & $\mathbf{i}$ & $-\mathbf{1}$ & $\mathbf{k}$ & $\mathbf{-j}$
\\ \hline
$\mathbf{j}$ & $\mathbf{j}$ & $\mathbf{-k}$ & $-\mathbf{1}$ & $\mathbf{i}$
\\ \hline
$\mathbf{k}$ & $\mathbf{k}$ & $\mathbf{j}$ & $\mathbf{-i}$ & $-\mathbf{1}$
\\ \hline
\end{tabular}%
\text{ }.
\end{equation*}%
Every $q\in \mathcal{H}$ can be simply written as $q=\func{Re}q+\func{Im}q$,
where $\func{Re}q=a_{0}$ and $\func{Im}q=\mathbf{i}a_{1}+\mathbf{j}a_{2}+%
\mathbf{k}a_{3}$ are called the real and imaginary parts, respectively. The
conjugate of the real quaternion $q$ is the quaternion denoted by $q^{\ast
}, $ and%
\begin{equation*}
q^{\ast }=\func{Re}q-\func{Im}q.
\end{equation*}%
This operation satisfies 
\begin{equation*}
(q^{\ast })^{\ast }=q,\text{ }(q_{1}+q_{2})^{\ast }=q_{1}^{\ast
}+q_{2}^{\ast },\text{ }(q_{1}q_{2})^{\ast }=q_{2}^{\ast }q_{1}^{\ast }
\end{equation*}%
for all $q_{1},$ $q_{2}\in \mathcal{H}.$ Also, any quaternion $q\in $ $%
\mathcal{H}$ can be written as 
\begin{equation*}
q=q_{1}+\mathbf{j}q_{2},
\end{equation*}%
where $q_{1},$ $q_{2}\in 
%TCIMACRO{\U{2102} }%
%BeginExpansion
\mathbb{C}
%EndExpansion
.$ The addition and multiplication of any two quaternions,

\begin{equation*}
q=q_{1}+\mathbf{j}q_{2},q^{\prime }=q_{1}^{\prime }+\mathbf{j}q_{2}^{\prime
},
\end{equation*}
are defined by%
\begin{equation*}
q+q^{\prime }=\left( q_{1}+q_{1}^{\prime }\right) +\mathbf{j(}%
q_{2}+q_{2}^{\prime })
\end{equation*}%
and%
\begin{equation*}
qq^{\prime }=[q_{1}q_{1}^{\prime }-(q_{2}^{\prime })^{\ast }q_{2}]+\mathbf{j}%
[(q_{2}^{\prime })^{\ast }q_{1}^{\ast }+q_{2}^{\ast }q_{1}^{\prime }].
\end{equation*}

\bigskip The norm of the quaternion $q$ is defined by 
\begin{equation*}
N(q)=q\overline{q}.
\end{equation*}%
Thus the inverse of a nonzero quaternion $q$ is given by 
\begin{equation*}
q^{-1}=\frac{\overline{q}}{N(q)}.
\end{equation*}%
For all $p,q\in \mathcal{H}$, we have 
\begin{eqnarray*}
N(pq) &=&N(p)N(q), \\
\left( pq\right) ^{-1} &=&q^{-1}p^{-1}.
\end{eqnarray*}

There are various types of quaternion sequences which are determined by
their components taken from different types of sequences and they have been
studied by many researchers. One of the well-known sequence is given by see 
\cite{Horadam1}. In \cite{Horadam1}, Horadam defined the $n^{th}$ Fibonacci
and Lucas quaternions as the quaternions whose components are Fibonacci and
Lucas numbers respectively. After that several authors were interested in
these structures and obtained some results, see \cite%
{Halici,Horadam2,Iakin,Iakin2,Iyer,Koshy,Lin}. Recently Cerda-Morales
considered the generalized Tribonacci sequence $\left\{ V_{n}\right\}
_{n\geq 0}$ defined by 
\begin{equation*}
V_{n}=rV_{n-1}+sV_{n-2}+tV_{n-3},\text{ }n\geq 3
\end{equation*}%
where $r,s,t$ are real numbers and $V_{0}=a,$ $V_{1}=b,$ $V_{2}=c$ are
arbitrary integers, see \cite{cerda}. For $r=s=t=1$ and $V_{0}=0,$ $V_{1}=1,$
$V_{2}=1$, the sequence $\left\{ V_{n}\right\} _{n\geq 0}$ is the well-known
Tribonacci sequence denoted by $\{T_{n}\}_{n},$ see \cite{Feinberg,Feng,
Kilic}. For $r=s=t=1$ and $V_{0}=3,$ $V_{1}=1,$ $V_{2}=3$, we obtain the
Tribonacci-Lucas sequence $\{K_{n}\}_{n},$ see \cite{Taskara}. The first few
Tribonacci numbers and Tribonacci Lucas numbers are given in the following
table:\ \ 
\begin{equation*}
\begin{tabular}{|r|r|r|r|r|r|r|r|r|r|r|r|r|r|r|}
\hline
$n$ & $0$ & $1$ & $2$ & $3$ & $4$ & $5$ & $6$ & $7$ & $8$ & $9$ & $10$ & $11$
& $12$ & $13$ \\ \hline
$T_{n}$ & $0$ & $1$ & $1$ & $2$ & $4$ & $7$ & $13$ & $24$ & $44$ & $81$ & $%
149$ & $274$ & $504$ & $927$ \\ \hline
$K_{n}$ & $3$ & $1$ & $3$ & $7$ & $11$ & $21$ & $39$ & $71$ & $131$ & $241$
& $443$ & $815$ & $1499$ & $2757$ \\ \hline
\end{tabular}%
\end{equation*}

The function 
\begin{equation*}
f(x)=a_{0}+a_{1}x+a_{2}x^{2}+\cdots +a_{n}x^{n}+\cdots
\end{equation*}%
is called the generating function for the sequence $\{a_{0},a_{1},a_{2},%
\ldots \}.$ The generating functions of the Tribonacci sequence $%
\{T_{n}\}_{n}$ \ and the Tribonacci-Lucas sequence $\{K_{n}\}_{n}$ are%
\begin{eqnarray*}
f(x) &=&\frac{x}{1-x-x^{2}-x^{3}}, \\
h(x) &=&\frac{3-2x-x^{2}}{1-x-x^{2}-x^{3}},
\end{eqnarray*}%
respectively. The Binet formulas of $T_{n}$ and $K_{n}$ are given as 
\begin{eqnarray}
T_{n} &=&\frac{\alpha ^{n+1}}{(\alpha -\beta )(\alpha -\gamma )}+\frac{\beta
^{n+1}}{(\beta -\alpha )(\beta -\gamma )}+\frac{\gamma ^{n+1}}{(\gamma
-\alpha )(\gamma -\beta )},  \TCItag{1} \\
K_{n} &=&\alpha ^{n}+\beta ^{n}+\gamma ^{n},  \notag
\end{eqnarray}%
respectively, where 
\begin{eqnarray*}
\alpha &=&\frac{1+\sqrt[3]{19+3\sqrt{33}}+\sqrt[3]{19-3\sqrt{33}}}{3} \\
\beta &=&\frac{1+w\sqrt[3]{19+3\sqrt{33}}+w^{2}\sqrt[3]{19-3\sqrt{33}}}{3} \\
\gamma &=&\frac{1+w^{2}\sqrt[3]{19+3\sqrt{33}}+w\sqrt[3]{19-3\sqrt{33}}}{3}%
,(w=\frac{-1+i\sqrt{3}}{2}),
\end{eqnarray*}%
see \cite{Taskara}.

In \cite{cerda}, we see a new type of quaternion whose coefficients are
generalized Tribonacci numbers as follows, 
\begin{equation*}
Q_{v,n}=V_{n}+V_{n+1}\mathbf{i}+V_{n+2}\mathbf{j}+V_{n+3}\mathbf{k},\text{ }%
n\geq 0.
\end{equation*}%
In this paper, we are interested in the quaternions with Tribonacci number
and Tribonacci-Lucas number components denoted by $Q_{n}$ and $\tilde{Q}_{n}$%
, respectively. We give some properties of these quaternions and obtain some
identities for them.

\section{Quaternions with Tribonacci Number Components}

For $n\geq 0,$ the $n^{\text{th}}$ Tribonacci quaternion $Q_{n}$ and $n^{%
\text{th}}$ Tribonacci-Lucas quaternion $\tilde{Q}_{n}$ are defined by 
\begin{equation*}
Q_{n}=T_{n}+\mathbf{i}T_{n+1}+\mathbf{j}T_{n+2}+\mathbf{k}T_{n+3}
\end{equation*}%
and 
\begin{equation*}
\tilde{Q}_{n}=K_{n}+\mathbf{i}K_{n+1}+\mathbf{j}K_{n+2}+\mathbf{k}K_{n+3}
\end{equation*}%
respectively, where $T_{n}$ and $K_{n}$ are $n^{\text{th}}$ Tribonacci and
Tribonacci-Lucas numbers.

Note that for $n\geq 0,$ 
\begin{equation*}
Q_{n+3}=Q_{n+2}+Q_{n+1}+Q_{n},
\end{equation*}%
and%
\begin{equation*}
\tilde{Q}_{n+3}=\tilde{Q}_{n+2}+\tilde{Q}_{n+1}+\tilde{Q}_{n}.
\end{equation*}%
The conjugate of the Tribonacci quaternion $Q_{n}$ is denoted by $%
Q_{n}^{\ast }$ and 
\begin{equation*}
Q_{n}^{\ast }=T_{n}-\mathbf{i}T_{n+1}-\mathbf{j}T_{n+2}-\mathbf{k}T_{n+3},
\end{equation*}%
and the conjugate of the Tribonacci-Lucas quaternion $\tilde{Q}_{n}$ is
denoted by $\tilde{Q}_{n}^{\ast }$ and 
\begin{equation*}
\tilde{Q}_{n}^{\ast }=K_{n}-\mathbf{i}K_{n+1}-\mathbf{j}K_{n+2}-\mathbf{k}%
K_{n+3}.
\end{equation*}%
Let $f(x)$ be a series in powers of $x$. Then by the symbol $[x^{n}]f(x)$ we
will mean the coefficient of $x^{n}$ in the series $f(x)$. Hence the norm of
the quaternion $Q_{n}$ is%
\begin{equation*}
Q_{n}Q_{n}^{\ast }=\dsum\limits_{i=0}^{3}T_{n+i}^{2}=[x^{n}]\frac{%
2(3+5x+4x^{2}-2x^{3}-x^{4}-x^{5})}{\left( 1-3x-x^{2}-x^{3}\right) \left(
1+x+x^{2}-x^{3}\right) }.
\end{equation*}

For $n\geq 2,$ let $A_{n}=T_{-n}$ and $B_{n}=K_{-n}.$ Then Tribonacci and
Tribonacci-Lucas sequences with negative indices are defined by the
following equations, see \cite{Taskara};%
\begin{eqnarray*}
A_{n} &=&-A_{n-1}-A_{n-2}+A_{n-3}\text{ \ \ \ }(A_{-1}=1,\text{ }%
A_{0}=A_{1}=0) \\
B_{n} &=&-B_{n-1}-B_{n-2}+B_{n-3}\text{ \ \ \ }(B_{-1}=1,\text{ }B_{0}=3,%
\text{ }B_{1}=-1).
\end{eqnarray*}%
Hence we can give the following definition.

\begin{definition}
The Tribonacci and Tribonacci-Lucas quaternions with negative subscripts are
defined by%
\begin{eqnarray*}
Q_{-n} &=&A_{n}+\mathbf{i}A_{n-1}+\mathbf{j}A_{n-2}+\mathbf{k}A_{n-3} \\
\tilde{Q}_{-n} &=&B_{n}+\mathbf{i}B_{n-1}+\mathbf{j}B_{n-2}+\mathbf{k}%
B_{n-3}.
\end{eqnarray*}
\end{definition}

The generating function and Binet formula for generalized Tribonacci
quaternions are given in \cite{cerda}. For the completeness of the paper, we
give the generating function and Binet formula for the Tribonacci
quaternions.

\begin{theorem}
The generating function for the Tribonacci quaternion $Q_{n}$ is 
\begin{equation*}
G(x)=\frac{x+\mathbf{i}+\mathbf{j}(1+x+x^{2})+\mathbf{k}(2+2x+x^{2})}{%
1-x-x^{2}-x^{3}}.
\end{equation*}
\end{theorem}

\begin{proof}
Let 
\begin{equation*}
G(x)=Q_{0}+Q_{1}x+Q_{2}x^{2}+\cdots +Q_{n}x^{n}+\cdots
\end{equation*}%
be the generating function of the Tribonacci quaternion $Q_{n}.$ Since the
orders of $Q_{n-1}$, $Q_{n-2}$ and $Q_{n-3}$ are $1$, $2$ and $3$ less than
the order of $Q_{n},$ respectively, find $xG\left( x\right) $, $x^{2}G\left(
x\right) $ and $x^{3}G(x):$

\begin{eqnarray*}
xG\left( x\right) &=&Q_{0}x+Q_{1}x^{2}+Q_{2}x^{3}+\cdots +Q_{n-1}x^{n}+\cdots
\\
x^{2}G\left( x\right) &=&Q_{0}x^{2}+Q_{1}x^{3}+Q_{2}x^{4}+\cdots
+Q_{n-2}x^{n}+\cdots . \\
x^{3}G\left( x\right) &=&Q_{0}x^{3}+Q_{1}x^{4}+Q_{2}x^{5}+\cdots
+Q_{n-3}x^{n}+\cdots .
\end{eqnarray*}%
Thus 
\begin{equation*}
G(x)=\frac{Q_{0}+x(Q_{1}-Q_{0})+x^{2}(Q_{2}-Q_{1}-Q_{0})}{1-x-x^{2}-x^{3}},
\end{equation*}%
and so 
\begin{equation*}
G(x)=\frac{x+\mathbf{i}+\mathbf{j}(1+x+x^{2})+\mathbf{k}(2+2x+x^{2})}{%
1-x-x^{2}-x^{3}}.
\end{equation*}
\end{proof}

\begin{theorem}
The Binet formulas for the Tribonacci and Tribonacci-Lucas quaternions are
given by%
\begin{eqnarray*}
Q_{n} &=&\frac{\alpha ^{n+1}}{(\alpha -\beta )(\alpha -\gamma )}\underline{%
\alpha }+\frac{\beta ^{n+1}}{(\beta -\alpha )(\beta -\gamma )}\underline{%
\beta }+\frac{\gamma ^{n+1}}{(\gamma -\alpha )(\gamma -\beta )}\underline{%
\gamma } \\
\tilde{Q}_{n} &=&\alpha ^{n}\underline{\alpha }+\beta ^{n}\underline{\beta }%
+\gamma ^{n}\underline{\gamma }
\end{eqnarray*}%
where $\underline{\alpha }=1+\mathbf{i}\alpha +\mathbf{j}\alpha ^{2}+\mathbf{%
k}\alpha ^{3},\underline{\beta }=1+\mathbf{i}\beta +\mathbf{j}\beta ^{2}+%
\mathbf{k}\beta ^{3}$ and $\underline{\gamma }=1+\mathbf{i}\gamma +\mathbf{j}%
\gamma ^{2}+\mathbf{k}\gamma ^{3}.$
\end{theorem}

\begin{proof}
Using the Binet formulas for $T_{n}$ and $K_{n}$ given in $(1)$ and the
definition of $Q_{n}$ and $\tilde{Q}_{n},$ we obtain the Binet's formula for 
$Q_{n}$ and $\tilde{Q}_{n}$ as follows, 
\begin{eqnarray*}
Q_{n} &=&\frac{\alpha ^{n+1}}{(\alpha -\beta )(\alpha -\gamma )}\underline{%
\alpha }+\frac{\beta ^{n+1}}{(\beta -\alpha )(\beta -\gamma )}\underline{%
\beta }+\frac{\gamma ^{n+1}}{(\gamma -\alpha )(\gamma -\beta )}\underline{%
\gamma } \\
\tilde{Q}_{n} &=&\alpha ^{n}\underline{\alpha }+\beta ^{n}\underline{\beta }%
+\gamma ^{n}\underline{\gamma }.
\end{eqnarray*}
\end{proof}

\section{Some Identities on Tribonacci Quaternions}

\subsection{Identities 1.}

\begin{equation*}
Q_{n}^{2}=2T_{n}Q_{n}-Q_{n}Q_{n}^{\ast },
\end{equation*}%
\begin{equation*}
Q_{n}+Q_{n}^{\ast }=2T_{n},
\end{equation*}%
\begin{equation*}
\tilde{Q}_{n}=Q_{n}+2Q_{n-1}+3Q_{n-2}.
\end{equation*}

\subsection{Identities 2.}

\begin{eqnarray*}
Q_{m+n} &=&Q_{m}K_{n}-Q_{m-n}C_{n}+Q_{m-2n}, \\
\tilde{Q}_{m+n} &=&\tilde{Q}_{m}K_{n}-\tilde{Q}_{m-n}C_{n}+\underline{C}%
_{2n-m},
\end{eqnarray*}%
where 
\begin{equation*}
C_{n}=\alpha ^{n}\beta ^{n}+\alpha ^{n}\gamma ^{n}+\beta ^{n}\gamma ^{n}
\end{equation*}
and 
\begin{equation*}
\underline{C}_{2n-m}=C_{2n-m}+\mathbf{i}C_{2n-m-1}+\mathbf{j}C_{2n-m-2}+%
\mathbf{k}C_{2n-m-3}.
\end{equation*}%
Another identity can be given as

\begin{equation*}
Q_{n+2m}=K_{m}Q_{n+m}-K_{-m}Q_{n}+Q_{n-2m}.
\end{equation*}

\subsection{Identity 3.}

For $n\geq 0$, $m\geq 3$ we have%
\begin{equation*}
Q_{n+m}=T_{m-2}Q_{n}+(T_{m-3}+T_{m-2})Q_{n+1}+T_{m-1}Q_{n+2}.
\end{equation*}

\subsection{Identities 4.}

Let $\widehat{S}_{n}=\sum\limits_{k=0}^{n}Q_{k}.$ Then we have 
\begin{equation*}
Q_{n}=\frac{1}{2}%
\begin{bmatrix}
\widehat{S}_{n}-\widehat{S}_{n-4}%
\end{bmatrix}%
\end{equation*}%
and for $n\geq 0$, $m\geq 5$ we have%
\begin{equation*}
\widehat{S}_{n+m}=-S_{m-3}\widehat{S}_{n}-S_{m-4}\widehat{S}_{n+1}-S_{m-5}%
\widehat{S}_{n+2}+S_{m-2}\widehat{S}_{n+3},
\end{equation*}%
where $S_{m}=\sum\limits_{k=0}^{m}T_{k}.$

\subsection{\protect\bigskip Identity 5.}

\QTP{Body Math}
\begin{equation*}
(Q_{n}Q_{n+4})^{2}+(2(Q_{n+1}+Q_{n+2})Q_{n+3})^{2}=(Q_{n}^{2}+2(Q_{n+1}+Q_{n+2})Q_{n+3})^{2}.
\end{equation*}

\QTP{Body Math}
$\bigskip $

\subsection{\protect\bigskip Identities 6.}

Let 
\begin{eqnarray*}
R_{n} &=&3T_{n+1}-T_{n}\text{ \ \ \ }(n\geq 0) \\
\tilde{R}_{n} &=&R_{n}+\mathbf{i}R_{n+1}+\mathbf{j}R_{n+2}+\mathbf{k}R_{n+3}
\end{eqnarray*}%
and 
\begin{eqnarray*}
U_{n} &=&T_{n-1}+T_{n-2}\text{ \ \ }(U_{0}=U_{1}=0,\text{ }n\geq 2) \\
\tilde{U}_{n} &=&U_{n}+\mathbf{i}U_{n+1}+\mathbf{j}U_{n+2}+\mathbf{k}U_{n+3}.
\end{eqnarray*}

Then we have%
\begin{equation}
\tilde{R}_{n+3}=\tilde{R}_{n+2}+\tilde{R}_{n+1}+\tilde{R}_{n},  \tag{1}
\end{equation}%
and%
\begin{equation*}
\tilde{U}_{n+3}=\tilde{U}_{n+2}+\tilde{U}_{n+1}+\tilde{U}_{n}.
\end{equation*}%
We also obtain the following identities: 
\begin{equation*}
Q_{n}^{2}-Q_{n-1}^{2}=\tilde{U}_{n+1}\tilde{U}_{n-1}\text{ \ \ }(n\geq 2).
\end{equation*}

\begin{equation*}
\tilde{U}_{n+1}^{2}+\tilde{U}_{n-1}^{2}=2(Q_{n-1}^{2}+Q_{n}^{2})\text{ \ \ }%
(n\geq 2).
\end{equation*}

\subsection{\protect\bigskip Identities 7.}

Now we will give some identities about the finite sum of various quaternions
that we obtained.%
\begin{equation*}
\sum\limits_{k=0}^{n}Q_{k}=\frac{Q_{n+2}+Q_{n}+Q_{0}-Q_{2}}{2}
\end{equation*}%
\begin{equation*}
\sum\limits_{k=0}^{n}Q_{2k}=\frac{Q_{2n+1}+Q_{2n}-(1+\mathbf{j}+2\mathbf{k})%
}{2}
\end{equation*}%
\begin{equation*}
\sum\limits_{k=0}^{n}Q_{2k+1}=\frac{Q_{2n+2}+Q_{2n+1}-(\mathbf{i+}2\mathbf{j}%
+3\mathbf{k})}{2}
\end{equation*}

\begin{eqnarray*}
\sum\limits_{k=0}^{n}Q_{3k} &=&\sum\limits_{k=0}^{3n-1}Q_{k}+Q_{0} \\
&=&\frac{Q_{3n+2}-Q_{3n}-(1-\mathbf{i}+\mathbf{j}+\mathbf{k})}{2}
\end{eqnarray*}%
\begin{equation*}
\sum\limits_{k=0}^{n}Q_{4k}=\frac{Q_{4n+2}+Q_{4n}-(1-\mathbf{i}+\mathbf{j}+%
\mathbf{k})}{4}
\end{equation*}

We also have 
\begin{equation*}
\sum\limits_{k=0}^{n}\tilde{U}_{n}=Q_{n+1}-(1+\mathbf{i}+\mathbf{j}+2\mathbf{%
k}).
\end{equation*}

\begin{equation*}
\sum\limits_{k=1}^{n}\tilde{Q}_{n}=2\tilde{U}_{n+2}+\tilde{U}_{n}-(3+4%
\mathbf{i}+7\mathbf{j}+14\mathbf{k})
\end{equation*}

\begin{equation*}
\sum\limits_{k=0}^{n}Q_{k}=\frac{\tilde{U}_{n+2}+\tilde{U}_{n+1}-(1+\mathbf{i%
}+3\mathbf{j}+5\mathbf{k})}{2}
\end{equation*}

\begin{equation*}
\sum\limits_{k=0}^{n}\tilde{R}_{k}=\frac{3\tilde{U}_{n+3}+2\tilde{U}_{n+2}-%
\tilde{U}_{n+1}-(2+8\mathbf{i}+12\mathbf{j}+22\mathbf{k})}{2}.
\end{equation*}

\begin{equation*}
\sum\limits_{k=0}^{n}\tilde{U}_{3k}=Q_{3n}-\mathbf{i.}
\end{equation*}

\begin{equation*}
\sum\limits_{k=0}^{n}\tilde{U}_{3k+1}=Q_{3n+1}-(1\mathbf{+k}).
\end{equation*}

\section{Proofs}

In order to keep this paper within reasonable length, we restricted
ourselves to a short selection. Thus we prove some identities using the
Binet formulae and mathematical induction. The other identities can be
proved similarly to the proofs.

\subsection{Proof of the Identities 1:}

We will give the proof of identity 
\begin{equation*}
Q_{n}^{2}=2T_{n}Q_{n}-Q_{n}Q_{n}^{\ast }.
\end{equation*}

We have 
\begin{equation*}
Q_{n}^{2}=T_{n}^{2}-T_{n+1}^{2}-T_{n+2}^{2}-T_{n+3}^{2}+2(\mathbf{i}%
T_{n}T_{n+1}+\mathbf{j}T_{n}T_{n+2}+\mathbf{k}T_{n}T_{n+3}).
\end{equation*}%
On the other hand since 
\begin{equation*}
Q_{n}Q_{n}^{\ast }=T_{n}^{2}+T_{n+1}^{2}+T_{n+2}^{2}+T_{n+3}^{2}
\end{equation*}%
and 
\begin{equation*}
2T_{n}Q_{n}=2T_{n}^{2}+2(\mathbf{i}T_{n}T_{n+1}+\mathbf{j}T_{n}T_{n+2}+%
\mathbf{k}T_{n}T_{n+3}),
\end{equation*}%
we get the result.

Now we will prove the identity%
\begin{equation*}
\tilde{Q}_{n}=Q_{n}+2Q_{n-1}+3Q_{n-2}.
\end{equation*}%
The Binet formula of the Tribonacci quaternion is given as%
\begin{equation*}
Q_{n}=\frac{\alpha ^{n+1}}{(\alpha -\beta )(\alpha -\gamma )}\underline{%
\alpha }+\frac{\beta ^{n+1}}{(\beta -\alpha )(\beta -\gamma )}\underline{%
\beta }+\frac{\gamma ^{n+1}}{(\gamma -\alpha )(\gamma -\beta )}\underline{%
\gamma }.
\end{equation*}%
Then we have 
\begin{eqnarray*}
Q_{n}+2Q_{n-1}+3Q_{n-2} &=&%
\begin{bmatrix}
\frac{\alpha ^{n+1}}{(\alpha -\beta )(\alpha -\gamma )}\underline{\alpha }+%
\frac{\beta ^{n+1}}{(\beta -\alpha )(\beta -\gamma )}\underline{\beta }+%
\frac{\gamma ^{n+1}}{(\gamma -\alpha )(\gamma -\beta )}\underline{\gamma }%
\end{bmatrix}%
+ \\
&&2%
\begin{bmatrix}
\frac{\alpha ^{n}}{(\alpha -\beta )(\alpha -\gamma )}\underline{\alpha }+%
\frac{\beta ^{n}}{(\beta -\alpha )(\beta -\gamma )}\underline{\beta }+\frac{%
\gamma ^{n}}{(\gamma -\alpha )(\gamma -\beta )}\underline{\gamma }%
\end{bmatrix}%
+ \\
&&3%
\begin{bmatrix}
\frac{\alpha ^{n-1}}{(\alpha -\beta )(\alpha -\gamma )}\underline{\alpha }+%
\frac{\beta ^{n-1}}{(\beta -\alpha )(\beta -\gamma )}\underline{\beta }+%
\frac{\gamma ^{n-1}}{(\gamma -\alpha )(\gamma -\beta )}\underline{\gamma }%
\end{bmatrix}
\\
&=&%
\begin{bmatrix}
\frac{\alpha ^{n+1}+2\alpha ^{n}+3\alpha ^{n-1}}{(\alpha -\beta )(\alpha
-\gamma )}%
\end{bmatrix}%
\underline{\alpha }+%
\begin{bmatrix}
\frac{\beta ^{n+1}+2\beta ^{n}+3\beta ^{n-1}}{(\beta -\alpha )(\beta -\gamma
)}%
\end{bmatrix}%
\underline{\beta }+ \\
&&%
\begin{bmatrix}
\frac{\gamma ^{n+1}+2\gamma ^{n}+3\gamma ^{n-1}}{(\gamma -\alpha )(\gamma
-\beta )}%
\end{bmatrix}%
\underline{\gamma } \\
&=&\alpha ^{n}%
\begin{bmatrix}
\frac{\alpha ^{2}+2\alpha +3}{\alpha (\alpha -\beta )(\alpha -\gamma )}%
\end{bmatrix}%
\underline{\alpha }+\beta ^{n}%
\begin{bmatrix}
\frac{\beta ^{2}+2\beta +3}{\beta (\beta -\alpha )(\beta -\gamma )}%
\end{bmatrix}%
\underline{\beta }+ \\
&&\gamma ^{n}%
\begin{bmatrix}
\frac{\gamma ^{2}+2\gamma +3}{\gamma (\gamma -\alpha )(\gamma -\beta )}%
\end{bmatrix}%
\underline{\gamma } \\
&=&\alpha ^{n}\underline{\alpha }+\beta ^{n}\underline{\beta }+\gamma ^{n}%
\underline{\gamma } \\
&=&\tilde{Q}_{n}.
\end{eqnarray*}

\subsection{Proof of the Identities 2:}

It is known that the Tribonacci numbers and Tribonacci-Lucas numbers satisfy
the equalities,%
\begin{eqnarray*}
T_{m+n} &=&T_{m}K_{n}-T_{m-n}C_{n}+T_{m-2n}, \\
K_{m+n} &=&K_{m}K_{n}-K_{m-n}C_{n}+C_{2n-m},
\end{eqnarray*}%
see \cite{Taskara}. Then we have 
\begin{eqnarray*}
Q_{m+n} &=&T_{m+n}+\mathbf{i}T_{m+n+1}+\mathbf{j}T_{m+n+2}+\mathbf{k}%
T_{m+n+3} \\
&=&(T_{m}K_{n}-T_{m-n}C_{n}+T_{m-2n})+\mathbf{i}%
(T_{m+1}K_{n}-T_{m+1-n}C_{n}+T_{m+1-2n})+ \\
&&\mathbf{j}(T_{m+2}K_{n}-T_{m+2-n}C_{n}+T_{m+2-2n})+ \\
&&\mathbf{k}(T_{m+3}K_{n}-T_{m+3-n}C_{n}+T_{m+3-2n}) \\
&=&(T_{m}+\mathbf{i}T_{m+1}+\mathbf{j}T_{m+2}+\mathbf{k}T_{m+3})K_{n}- \\
&&(T_{m-n}+\mathbf{i}T_{m-n+1}+\mathbf{j}T_{m-n+2}+\mathbf{k}T_{m-n+3})C_{n}+
\\
&&(T_{m-2n}+\mathbf{i}T_{m-2n+1}+\mathbf{j}T_{m-2n+2}+\mathbf{k}T_{m-2n+3})
\\
&=&Q_{m}K_{n}-Q_{m-n}C_{n}+Q_{m-2n}.
\end{eqnarray*}

For all $n$ and $m,$ Tribonacci and Tribonacci-Lucas sequences also satisfy
the following equality,%
\begin{equation*}
T_{n+2m}=K_{m}T_{n+m}-K_{-m}T_{n}+T_{n-2m},
\end{equation*}%
see \cite{Howard}. Similarly we obtain the identity%
\begin{equation*}
Q_{n+2m}=K_{m}Q_{n+m}-K_{-m}Q_{n}+Q_{n-2m}.
\end{equation*}

\subsection{Proof of the Identity 3:}

For $m=3,$ we have%
\begin{eqnarray*}
Q_{n+3} &=&Q_{n}+Q_{n+1}+Q_{n+2} \\
&=&T_{1}Q_{n}+(T_{0}+T_{1})Q_{n+1}+T_{2}Q_{n+2}.
\end{eqnarray*}%
Suppose the equality holds for all $m\leq k.$ For $m=k+1,$ we have 
\begin{eqnarray*}
Q_{n+k+1} &=&Q_{n+k}+Q_{n+k-1}+Q_{n+k-2} \\
&=&T_{k-2}Q_{n}+(T_{k-3}+T_{k-2})Q_{n+1}+T_{k-1}Q_{n+2}+T_{k-3}Q_{n}+ \\
&&(T_{k-4}+T_{k-3})Q_{n+1}+T_{k-2}Q_{n+2}+T_{k-4}Q_{n}+(T_{k-5}+T_{k-4})Q_{n+1}+T_{k-3}Q_{n+2}
\\
&=&(T_{k-2}+T_{k-3}+T_{k-4})Q_{n}+ \\
&&(T_{k-3}+T_{k-2}+T_{k-4}+T_{k-3}+T_{k-5}+T_{k-4})Q_{n+1}+ \\
&&(T_{k-1}+T_{k-2}+T_{k-3})Q_{n+2} \\
&=&T_{k-1}Q_{n}+(T_{k-2}+T_{k-1})Q_{n+1}+T_{k}Q_{n+2}.
\end{eqnarray*}

By induction on $m,$ we get the result.

\subsection{Proof of the Identities 4:}

Since 
\begin{eqnarray*}
\widehat{S}_{n} &=&Q_{0}+Q_{1}+\cdots +Q_{n} \\
&=&Q_{0}+Q_{1}+\cdots +Q_{n-4}+Q_{n-3}+Q_{n-2}+Q_{n-1}+Q_{n} \\
&=&\widehat{S}_{n-4}+Q_{n}+Q_{n} \\
&=&\widehat{S}_{n-4}+2Q_{n},
\end{eqnarray*}%
we obtain that 
\begin{equation*}
Q_{n}=\frac{1}{2}%
\begin{bmatrix}
\widehat{S}_{n}-\widehat{S}_{n-4}%
\end{bmatrix}%
.
\end{equation*}%
For the other identity the proof will be done by induction on $n$ and $m.$
First we will prove the identity%
\begin{equation*}
\widehat{S}_{n+5}=-2\widehat{S}_{n}-\widehat{S}_{n+1}+4\widehat{S}_{n+3}.
\end{equation*}

For $n=0,$ we have 
\begin{eqnarray*}
\widehat{S}_{5} &=&Q_{0}+Q_{1}+Q_{2}+Q_{3}+Q_{4}+Q_{5} \\
&=&Q_{0}+Q_{1}+Q_{2}+Q_{3}+(Q_{1}+Q_{2}+Q_{3})+(Q_{2}+Q_{3}+Q_{4}) \\
&=&Q_{0}+Q_{1}+Q_{2}+Q_{3}+(Q_{1}+Q_{2}+Q_{3})+(Q_{2}+Q_{3}+Q_{1}+Q_{2}+Q_{3})
\\
&=&Q_{0}+3Q_{1}+4Q_{2}+4Q_{3} \\
&=&-2Q_{0}-Q_{0}-Q_{1}+4(Q_{0}+Q_{1}+Q_{2}+Q_{3}) \\
&=&-2\widehat{S}_{0}-\widehat{S}_{1}+4\widehat{S}_{3}.
\end{eqnarray*}

Suppose the equality holds for $n=k,$ that is, 
\begin{equation*}
\widehat{S}_{k+5}=-2\widehat{S}_{k}-\widehat{S}_{k+1}+4\widehat{S}_{k+3}.
\end{equation*}%
For $n=k+1,$ we have 
\begin{eqnarray*}
\widehat{S}_{k+6} &=&\widehat{S}_{k+5}+Q_{k+6} \\
&=&-2\widehat{S}_{k}-\widehat{S}_{k+1}+4\widehat{S}_{k+3}+Q_{k+6} \\
&=&-2\widehat{S}_{k}-\widehat{S}_{k+1}+4\widehat{S}%
_{k+3}+(Q_{k+3}+Q_{k+4}+Q_{k+5}) \\
&=&-2\widehat{S}_{k}-\widehat{S}_{k+1}+4\widehat{S}%
_{k+3}+(Q_{k+4}-Q_{k+2}-Q_{k+1}+Q_{k+4}+Q_{k+4}-Q_{k+2}-Q_{k+1}+Q_{k+4}) \\
&=&-2\widehat{S}_{k+1}-\widehat{S}_{k+2}+4\widehat{S}_{k+4}.
\end{eqnarray*}

So the equality holds for all $n\geq 0.$\bigskip

For $m=5,$ we have%
\begin{eqnarray*}
\widehat{S}_{n+5} &=&-2\widehat{S}_{k+1}-\widehat{S}_{k+2}+4\widehat{S}%
_{k+4}. \\
&=&-S_{2}\widehat{S}_{n}-S_{1}\widehat{S}_{n+1}-S_{0}\widehat{S}_{n+2}+S_{3}%
\widehat{S}_{n+3}.
\end{eqnarray*}%
Suppose the equality holds for $m=r,$ that is 
\begin{equation*}
\widehat{S}_{n+r}=-S_{r-3}\widehat{S}_{n}-S_{r-4}\widehat{S}_{n+1}-S_{r-5}%
\widehat{S}_{n+2}+S_{r-2}\widehat{S}_{n+3}.
\end{equation*}

For $m=r+1,$%
\begin{eqnarray*}
\widehat{S}_{n+r+1} &=&\widehat{S}_{n+r}+Q_{n+r+1} \\
&=&-S_{r-3}\widehat{S}_{n}-S_{r-4}\widehat{S}_{n+1}-S_{r-5}\widehat{S}%
_{n+2}+S_{r-2}\widehat{S}_{n+3}+(Q_{n+r-2}+Q_{n+r-1}+Q_{n+r}) \\
&=&-S_{r-2}\widehat{S}_{n}-S_{r-3}\widehat{S}_{n+1}-S_{r-4}\widehat{S}%
_{n+2}+S_{r-1}\widehat{S}_{n+3}.
\end{eqnarray*}

By induction on $m,$ we get the result.

\subsection{Proof of the Identity 5:}

We have%
\begin{eqnarray*}
Q_{n}^{2} &=&(Q_{n+3}-(Q_{n+1}+Q_{n+2}))^{2} \\
&=&Q_{n+3}^{2}+(Q_{n+1}+Q_{n+2})^{2}-2(Q_{n+1}+Q_{n+2})Q_{n+3}
\end{eqnarray*}%
and this gives 
\begin{equation*}
Q_{n}^{2}+2(Q_{n+1}+Q_{n+2})Q_{n+3}=Q_{n+3}^{2}+(Q_{n+1}+Q_{n+2})^{2}.
\end{equation*}%
Thus 
\begin{eqnarray*}
(Q_{n}^{2}+2(Q_{n+1}+Q_{n+2})Q_{n+3})^{2}
&=&Q_{n+3}^{4}+(Q_{n+1}+Q_{n+2})^{4}+2((Q_{n+1}+Q_{n+2})Q_{n+3})^{2} \\
&=&(Q_{n+3}^{2}-(Q_{n+1}+Q_{n+2})^{2})^{2}+(2(Q_{n+1}+Q_{n+2})Q_{n+3})^{2}.
\end{eqnarray*}

Here 
\begin{eqnarray*}
(Q_{n+3}^{2}-(Q_{n+1}+Q_{n+2})^{2})^{2}
&=&(Q_{n+3}-(Q_{n+1}+Q_{n+2}))^{2}+(Q_{n+3}+(Q_{n+1}+Q_{n+2}))^{2} \\
&=&Q_{n}^{2}Q_{n+4}^{2}.
\end{eqnarray*}%
Substituting this gives the result.

\subsection{\protect\bigskip Proof of the Identity 6:}

For $n\geq 2,$ we have

\begin{eqnarray*}
Q_{n}^{2}-Q_{n-1}^{2} &=&(Q_{n}+Q_{n-1})(Q_{n}-Q_{n-1}) \\
&=&[(T_{n}+T_{n-1})+\mathbf{i}(T_{n+1}+T_{n})+\mathbf{j}(T_{n+2}+T_{n+1})+%
\mathbf{k}(T_{n+3}+T_{n+2})]\times \\
&&\lbrack (T_{n}-T_{n-1})+\mathbf{i}(T_{n+1}-T_{n})+\mathbf{j}%
(T_{n+2}-T_{n+1})+\mathbf{k}(T_{n+3}-T_{n+2})] \\
&=&[(T_{n}+T_{n-1})+\mathbf{i}(T_{n+1}+T_{n})+\mathbf{j}(T_{n+2}+T_{n+1})+%
\mathbf{k}(T_{n+3}+T_{n+2})]\times \\
&&\lbrack (T_{n-2}+T_{n-3})+\mathbf{i}(T_{n-1}+T_{n-2})+\mathbf{j}%
(T_{n}+T_{n-1})+\mathbf{k}(T_{n+1}+T_{n})] \\
&=&\tilde{U}_{n+1}\tilde{U}_{n-1}.
\end{eqnarray*}

\subsection{Proof of the Identities 7:}

\bigskip We will show the identity 
\begin{equation*}
\sum\limits_{k=0}^{n}Q_{k}=\frac{Q_{n+2}+Q_{n}+Q_{0}-Q_{2}}{2}.
\end{equation*}%
The others can be done similarly. The proof can be done by induction on $n.$
For $n=0$ we have 
\begin{equation*}
Q_{0}=\frac{Q_{2}+Q_{0}+Q_{0}-Q_{2}}{2}.
\end{equation*}%
So equality holds for $n=0.$ Assume it is true for $n=m,$ that is, 
\begin{equation*}
\sum\limits_{k=0}^{m}Q_{k}=\frac{Q_{m+2}+Q_{m}+Q_{0}-Q_{2}}{2}.
\end{equation*}

For $n=m+1,$ we have 
\begin{equation*}
\sum\limits_{k=0}^{m+1}Q_{k}=\sum\limits_{k=0}^{m}Q_{k}+Q_{m+1}.
\end{equation*}%
By induction hypothesis we can write%
\begin{eqnarray*}
\sum\limits_{k=0}^{m}Q_{k}+Q_{m+1} &=&\frac{Q_{m+2}+Q_{m}+Q_{0}-Q_{2}}{2}%
+Q_{m+1} \\
&=&\frac{Q_{m+2}+Q_{m}+Q_{0}-Q_{2}+2Q_{m+1}}{2} \\
&=&\frac{Q_{m+2}+Q_{m+1}+Q_{m}+Q_{m+1}+Q_{0}-Q_{2}}{2} \\
&=&\frac{Q_{m+3}+Q_{m+1}+Q_{0}-Q_{2}}{2}.
\end{eqnarray*}

Hence we obtain that 
\begin{equation*}
\sum\limits_{k=0}^{m+1}Q_{k}=\frac{Q_{m+3}+Q_{m+1}+Q_{0}-Q_{2}}{2}.
\end{equation*}

This shows that equality holds for all $n\geq 0.$

\section{A Note on the Tribonacci Quaternions}

We consider the Tribonacci and Tribonacci-Lucas quaternions. These
quaternions can be written as 
\begin{eqnarray*}
Q_{n} &=&T_{n}+A\text{ } \\
\tilde{Q}_{n} &=&K_{n}+B
\end{eqnarray*}%
where $A=\func{Im}Q_{n}$ and $B=\func{Im}\tilde{Q}_{n}.$ Let 
\begin{equation*}
\mathcal{QM=\{}Q_{n}:Q_{n}\text{ is the }n^{th}\text{ Tribonacci quaternion}%
\}
\end{equation*}%
and $\mathcal{M}$ is the set of $2\times 2$ matrices with entries from $%
%TCIMACRO{\U{2102} }%
%BeginExpansion
\mathbb{C}
%EndExpansion
$ of the form:%
\begin{equation*}
\mathcal{M=}\left\{ X_{n}:X_{n}=%
\begin{bmatrix}
z & -w \\ 
\overline{w} & \overline{z}%
\end{bmatrix}%
\text{ ; }z,w\in 
%TCIMACRO{\U{2102} }%
%BeginExpansion
\mathbb{C}
%EndExpansion
\right\} .
\end{equation*}%
Then each matrix can be decomposed into a vector space representation with
four basis elements. Let $\Phi $ be the following map:%
\begin{eqnarray*}
\Phi &:&\mathcal{QM\rightarrow M} \\
Q_{n} &\mapsto &X_{n}=%
\begin{bmatrix}
T_{n}+iT_{n+1} & -T_{n+2}-iT_{n+3} \\ 
T_{n+2}-iT_{n+3} & T_{n}-iT_{n+1}%
\end{bmatrix}%
.
\end{eqnarray*}%
Then it can be easily show that $\Phi $ is an isomorphism. Thus we can write 
\begin{equation*}
X_{n}=T_{n}E+T_{n+1}I+T_{n+2}J+T_{n+3}K
\end{equation*}%
where 
\begin{equation*}
E=%
\begin{bmatrix}
1 & 0 \\ 
0 & 1%
\end{bmatrix}%
,\text{ }I=%
\begin{bmatrix}
i & 0 \\ 
0 & -i%
\end{bmatrix}%
,\text{ }J=%
\begin{bmatrix}
0 & -1 \\ 
1 & 0%
\end{bmatrix}%
,\text{ }K=%
\begin{bmatrix}
0 & -i \\ 
i & 0%
\end{bmatrix}%
.
\end{equation*}%
Since $\det (X_{n})\neq 0,$ $X_{n}$ is an invertible matrix and its inverse
is in $\mathcal{M}$.

\end{document}